\documentclass[letterpaper, 10 pt, journal,twoside]{IEEEtran}
\IEEEoverridecommandlockouts   

\usepackage{xcolor} 
\usepackage{amssymb} 
\usepackage{amsmath} 
\usepackage{bigints} 
\usepackage{braket} 
\usepackage{physics} 
\usepackage{graphicx} 
\usepackage[nottoc]{tocbibind} 
\usepackage{url}
\newtheorem{theorem}{Theorem}

\newtheorem{lemma}{Lemma}

\newtheorem{definition}{Definition}

\newtheorem{remark}{Remark}

\newcommand{\iu}{\mathrm{i}\mkern1mu} 
\newcommand{\uc}{\boldsymbol{\mathrm{u}}} 

\def \C{\mathbb{C}}

\def \N{\mathbb{N}}

\def \R{\mathbb{R}}

\def \1{\mathds{1}}

\usepackage{imakeidx}

\usepackage{enumerate} 
\usepackage{caption}
\usepackage{tikz-cd}
\usepackage{tikz} 
\usetikzlibrary{chains, positioning, shapes.symbols,shapes,arrows}
\definecolor{darkblue}{RGB}{9,72,90}
\definecolor{lightblue}{RGB}{124,184,201}
\usepackage{pgf} 

\title{\LARGE \bf
Observed quantum particles system with graphon interaction
}

\author{Sofiane Chalal$^{1}$,  Nina H. Amini$^{1}$, Gaoyue Guo $^{2},$ and Hamed Amini $^{3}$ 
\thanks{This work is supported by the Agence Nationale de la Recherche projects Q-COAST ANR- 19-CE48-0003 and IGNITION ANR-21-CE47- 0015.}
\thanks{$^{1}$ CNRS, L2S, CentraleSupélec, Université Paris-Saclay.
        {\tt\tiny sofiane.chalal@centralesupelec.fr; nina.amini@centralesupelec.fr}.}%
\thanks{$^{2}$ MICS, CentraleSupélec, Université Paris-Saclay.
        {\tt\tiny gaoyue.guo@centralesupelec.fr}.}%
\thanks{$^{3}$ Department of Industrial and Systems Engineering, University of Florida, Gainesville, FL, USA.
        {\tt\tiny aminil@ufl.edu}.}}

\begin{document}

\maketitle
\thispagestyle{empty}
\pagestyle{empty}

\begin{abstract}
In this paper, we consider a system of heterogeneously interacting quantum particles subject to indirect continuous measurement. The interaction is assumed to be of the mean-field type. We derive a new limiting quantum graphon  system, prove the well-posedness of this system, and establish a stability result.
\end{abstract}

\section{Introduction }
The mean-field approach has been widely used to deduce simplified models representing large number of interacting systems. However, the traditional mean-field limit may not be applicable to many networks. In such scenarios, a continuum model that incorporates the concept of a graphon could be utilized. This model captures situations where a large number of particles interact in diverse ways. It has been shown that for a sequence of dense graphs (representing interactions), as the number of nodes approaches infinity, the limiting equation can be described by a graphon \cite{lovasz12large}. Well-posedness, continuity and stability of such systems have been established in \cite{bayraktar23graphon}. A propagation of chaos is proved in  \cite{bayraktar23propagation} and the application of this result in the theory of graphon mean-field games has been studied, see also \cite{caines19graphon,caines21graphon,amini23graphon,lacker22label} regarding the later direction. 

The foundational principle of the law of large numbers for a system comprising interacting quantum particles was initially established by Spohn in \cite{spohn1980kinetic}, where the resulting limit equation is commonly known as the non-linear Schr\"odinger equation or the Hartree equation. Subsequently, a stochastic variant of this framework was introduced in a series of papers by Kolokoltsov \cite{kolokoltsov2021law,kolokoltsov2022qmfg,kolokoltsov2021qmfgcounting}, culminating in the derivation of a novel stochastic Lindblad equation for homogeneous interacting quantum particles. For imperfect measurements, we refer to \cite{chalal23mean} where an extension of the previous results is provided. Furthermore the application of the limit equation in quantum feedback control is proposed in \cite{chalal23mean}.

In contrast to classical mean-field  theory, the primary challenges arise due to the entanglements among quantum particles, negating the validity of employing the classical empirical measure. Additionally, the act of measurement itself exerts a consequential back-action on the system's dynamics. To address these issues, the aforementioned papers consider quantum Belavkin filtering \cite{belavkin79optimal,belavkin83theory,belavkin87non}, employing indirect measurement frameworks.

The aim of this paper is to initiate the mean-field study for quantum particles with heterogeneous interaction described by graphon. This type of interaction corresponds to physical situation considered in the literature, see e.g. \cite{tindall22,searle24thermo} in the context of physical many body systems. The main result of this paper is to establish the existence, uniqueness, and stability properties for the limiting quantum graphon particle system in both cases of homodyne and counting detection. 
Our future work \cite{amini25quantum} will further explore these concepts from a mathematical point of view. 

The paper is structured as follows. We first present the theory of graphons in Section \ref{sec:graphon}, followed by a brief description of the set-up for interacting quantum particles in Section \ref{sec:interacting}. Section \ref{sec:result} is devoted to the main results on quantum graphon particles. In Section \ref{sec:numerical}, we give a numerical example. Finally, we conclude in Section \ref{sec:conc}. 

\subsection{Notations}
Let $T > 0$ be a fixed time horizon. Given a Polish space $S$, let $\mathcal{C}([0,T],S)$ and $\mathcal{D}([0,T],S)$ denote the spaces of continuous and càdlàg functions from $[0,T]$ to $S$, respectively, equipped with the topology of uniform convergence. For any measurable subset $J \subset \R$, we denote by $\mathfrak{B}(J) $ the space of Borel subsets of  $J$. We fix  $\mathbb{H}$ as the Hilbert space, for $k \geq 1 $ $\mathbb{H}^{\otimes k} := \overbrace{\mathbb{H}\otimes\dots\otimes\mathbb{H}}^{k \text{ times}}$ is the Hilbert space of $k$ distinguishable quantum particles of same species.
Let $\mathcal{B}_1(\mathbb{H})$ be the set of linear bounded operator on $\mathbb{H}$. For every $ {O} \in \mathcal{B}_1(\mathbb{H})$, denote by ${{O}}^{\dag}$ its adjoint operator. 
For any   ${O}_{a},{O}_{b} \in \mathcal{B}_1(\mathbb{H})$, set $[{O}_{a},{O}_{b}] := {O}_{a}{O}_{b} - {O}_{b}{O}_{a}$  and $\{{O}_{a},{O}_{b}\} := {O}_{a}{O}_{b} + {O}_{b}{O}_{a}$. For any operator ${O} \in \mathcal{B}_1(\mathbb{H})$ and for $j \in \{1,\dots,N\}$ denote 
$\mathbf{O}_{j} := \mathbf{I}\otimes\dots\otimes{O}\otimes\dots\otimes\mathbf{I}$ the operator on $\mathcal{B}_{1}(\mathbb{H}^{\otimes N})$ that acts only on the $j$-th Hilbert space.
For any operator $B \in \mathcal{B}_{1}(\mathbb{H}\otimes\mathbb{H}) $ and for $j,k \in \{1,\dots,N \}$ denote $\mathbf{B}_{jk}$ the operator on $\mathcal{B}_{1}(\mathbb{H}^{\otimes N})$ that acts only on $j$-th and $k$-th Hilbert space.
We will use $C$ to denote various constants in the paper and $C(p)$ to emphasize the dependence on some parameter $p$. Their values may change from line to line.
For $O \in \mathcal{B}_{1}(\mathbb{H})$ we note $\|O\| := \|O\|_{F} = \sqrt{\tr(OO^{\dag})}$.

We consider a non-oriented graph $G_N = (V_N,E_N)$ of size $N$, where $V_N $ denotes the vertices labeled as follows: $ V_N = \{\frac{1}{N},\dots,\frac{N}{N}\}$, and $E_N$ represents the edges between vertices and can be weighted  between $0$ and $1$. 

The following matrices form the Pauli matrices :
\begin{align*}
{\sigma}_x &:= \begin{pmatrix}0 & 1 \\ 1 & 0\end{pmatrix}, \quad
{\sigma}_y := \begin{pmatrix}0 & -\iu \\ \iu & 0\end{pmatrix}, \quad
{\sigma}_z := \begin{pmatrix}1 & 0 \\ 0 & -1\end{pmatrix}.
\end{align*}

\section{Generalities about Graphons }
\label{sec:graphon}
We start by providing a brief description of graphon theory, which allows to define the limit of a sequence of dense graphs, see \cite{lovasz12large} for more details. 

Let $I := [0,1]$. Denote by \(\mathcal{W}\) the space of all bounded symmetric measurable functions \(\mathrm{W} : I^2 \rightarrow \mathbb{R}\). The elements of $W \in \mathcal{W}$ will be called kernels. Let \(\mathcal W_0\) be the subset of $\mathcal{W}$ such that \(0 \leq \mathrm{W} \leq 1\), where elements of \(\mathcal{W}_0\) will be called graphons. 

Here we give an important notion in the
graphon theory which defines the operator associated to any kernel.
\begin{definition}[Kernel of graphon]
To a graphon $\mathrm{W}$ we associate the operator $ \mathcal{T}_{\mathrm{W}} : L^{\infty}(I) \to L^1(I)$ as follows for any $ \phi \in L^{\infty}(I)$
$$ \mathcal{T}_{\mathrm{W}}(\phi)(u) := \int_{0}^{1}\mathrm{W}(u,v)\phi(v)\mathrm{d}v.$$
\end{definition}
\begin{definition}[Norms]
Here we give two important norms which are used to quantify distance.
\begin{itemize}
\item Cut norm :
For any kernel $\mathrm{W} \in \mathcal{W}$, we define   the cut norm 
$$ \|\mathrm{W}\|_{\square} := \sup_{\mathrm{B}_1,\mathrm{B}_2 \in \mathfrak{B}(I)}\Big|\int_{\mathrm{B}_{1}\times\mathrm{B}_{2}}\mathrm{W}(u,v)\mathrm{d}u\mathrm{d}v\Big|.$$

This optima is attained, and it is equal to 
$$\sup_{\phi_1,\phi_2 : I \to I}\Big|\int_{I^2}\phi_{1}(u)\phi_{2}(v)\mathrm{W}(u,v)\mathrm{d}u\mathrm{d}v\Big|.$$
\item Operator norm : 
For an operator $\mathcal{T}_{\mathrm{W}}$ associated to kernel $\mathrm{W} \in \mathcal{W}$ we define the operator norm 
\begin{align*}
    \|\mathcal{T}_{\mathrm{W}}\|_{\text{op}} := \sup_{\|\phi\|_{\infty} \leq 1}\int_{I}\Big|\int_{I}\mathrm{W}(u,v)\phi(v)\mathrm{d}v\Big|\mathrm{d}u.
\end{align*}
\end{itemize}
\end{definition}
\medskip
We have the following inequalities between different norms, for more details, see \cite[lemma 8.11]{lovasz12large}.
\begin{lemma}
For every kernel $\mathrm{W} \in \mathcal{W}$, we have 
$$\|\mathrm{W}\|_{\square} \leq \|\mathcal{T}_{\mathrm{W}}\|_{\text{op}} \leq 4\|\mathrm{W}\|_{\square}.$$\label{estimatecutop}
\end{lemma}

One can establish a natural relationship between the adjacency matrix of a graph and graphon function.

\begin{definition}[Step Kernel]
Given a graph \(G_N = (V_{N},E_{N})\), one can define a step function \(\mathrm{W}^{G_N} : I^2 \rightarrow \mathbb{R}\) that associates with the adjacency matrix \((J_{pq}^N)_{p, q \in V_N}\). This association is constructed by partitioning the interval $I$ into $N$ equal subintervals \(I_p = \left(p-\frac{1}{N}, p\right]\) for \(p = \frac{1}{N}, \frac{2}{N}, \dots, \frac{N}{N}\), such that
\begin{align*}
    \mathrm{W}^{G_N}(u,v) = J_{pq}^N, \quad \text{for } (u,v) \in I_p \times I_q.
\end{align*}
\end{definition}


The following  picture captures the essential idea of graphon as limiting object.
\newline
\newline
\begin{tikzpicture}[line width=0.38pt]
    \begin{scope}[xshift=-0.5cm, yshift=-3.5, scale = 0.5]
        \foreach \i in {1,...,3}{
            \foreach \j in {\i,...,3}{
                \draw (\i,1) -- (\j,2);
            }
            \draw[fill=lightgray] (\i,2) circle [radius=1mm];
            \draw[fill=lightgray] (\i,1) circle [radius=1mm];
        }
    \end{scope}
    
    \begin{scope}[xshift=-4.3cm, yshift=-2cm, scale = 0.63]
        \foreach \i in {0,...,5}{
            \draw (6+\i/2,0) -- (6+\i/2,3);
            \draw (6, \i/2) -- (9, \i/2);

            \ifnum \i<3
                \filldraw (9-\i/2,3) rectangle (8.5-\i/2,1.5+\i/2);
                \filldraw (6+\i/2,0) rectangle (6.5+\i/2,1.5-\i/2);
            \fi
        }
        \draw (6,0) rectangle (9,3);
    \end{scope}

    \begin{scope}[xshift=2cm,yshift=-3.5, scale = 0.5]
        \foreach \i in {1,...,5}{
            \foreach \j in {\i,...,5}{
                \draw (\i,1) -- (\j,2);
            }
            \draw[fill=lightgray] (\i,2) circle [radius=1mm];
            \draw[fill=lightgray] (\i,1) circle [radius=1mm];
        }
    \end{scope}
    \begin{scope}[xshift=-0.1cm, yshift=-2.0cm, scale = 0.38]
        \foreach \i in {0,...,9}{
            \draw (7+\i/2,0) -- (7+\i/2,5);
            \draw (7, \i/2) -- (12, \i/2);

            \ifnum \i<5
                \filldraw (12-\i/2,5) rectangle (11.5-\i/2,2.5+\i/2);
                \filldraw (7+\i/2,0) rectangle (7.5+\i/2,2.5-\i/2);
            \fi
        }
        \draw (7,0) rectangle (12,5);
    \end{scope}

    \begin{scope}[xshift = 2.77cm, yshift = -1.7cm, scale = 0.26]
    
    \draw (10.3,-1) rectangle (17.5,6);

    \draw[fill = black] (17.5,6) -- (13.9,6) -- (17.5,2.5) -- cycle;

    \draw[fill = black] (10.3,-1) -- (13.9,-1) -- (10.3,2.5);
    \end{scope}
\end{tikzpicture}

\section{Interacting quantum particles}
\label{sec:interacting}
Our analysis begins by defining a Borel set $\mathfrak{X}$ equipped with a Borel measure $\mu$. The pair $(\mathfrak{X},\mu)$ constitutes a Borel space, from which we construct a Hilbert space for one particle $\mathbb{H} := L^2(\mathfrak{X};\C)$.

Consider a system of $N$ particles placed on a graph $G_N = (V_N,E_N)$ of  pairwise interaction whose intensity is given by the weights of the edges. The state of the whole system is represented by self-adjoint positive normalized trace-class operators  namely : 

$$ \mathcal{S}(\mathbb{H}^{G_N}) := \Big\{ \rho \in \mathcal{B}_1(\mathbb{H}^{G_N});\; \rho \geq 0, \tr(\rho) = 1, \rho = \rho^{\dag} \Big\},$$ 
where $ \mathbb{H}^{G_N} := \underset{|V_N|}{\bigotimes}\mathbb{H} = L^2(\mathfrak{X}^{N};\C).$

In addition to the pairwise Hamiltonian, a free Hamiltonian is attached to each particle such that the total Hamiltonian of the system is described by the following formula:
\begin{align}
    \mathbf{H}^{} &= \frac{1}{N}\sum_{q \in V_N}\sum_{p > q} \xi^{N}_{pq}\mathbf{A}_{pq} + \sum_{q \in V_N}\mathbf{\tilde{H}}_{q}.\label{hamilt}
\end{align}
\begin{itemize}
\item The operator $\mathbf{\tilde{H}}_q = \mathbf{\tilde{H}}_{q}^{\dag}$ is the free Hamiltonian associated to the particle $q$.

\item The operator $\mathbf{A}_{pq}$ represents the pairwise interaction between particles $q$ and $p$. The corresponding operator $A$ is symmetric, self-adjoint integral operator with Hilbert-Schmidt kernel $a$ such that:
\begin{align*}
&\|a\|^{2}_{HS} \!\!=\!\! \int_{\mathfrak{X}^4}\!\!|a(x,y;\tilde{x},\tilde{y})|^{2}\mu(\mathrm{d}x)\mu(\mathrm{d}y)\mu(\mathrm{d}\tilde{x})\mu(\mathrm{d}\tilde{y}) \!\!< \!\!\infty,\\
&a(x,y;\tilde{x},\tilde{y})\! = \!a(y,x;\tilde y,\tilde{x}),
a(x,y;\tilde{x},\tilde{y})\! = \!\overline{a(\tilde{x},\tilde{y};x,y)},
\end{align*}
\begin{align*} 
\quad &A : L^2(\mathfrak{X}^2;\C) \to
L^2(\mathfrak{X}^2;\C),\\
  Af(x,y) &:= \int_{\mathfrak{X}^2}a(x,y;\tilde{x},\tilde{y})f(\tilde{x},\tilde{y})\mu(\mathrm{d}\tilde{x})\mu(\mathrm{d}\tilde{y}).
\end{align*}

\item The coefficient $\xi_{pq}^{N}$, derived from the step kernel function $\mathrm{W}^{G_N}$, can be constructed in two ways:
\begin{itemize}
    \item $\xi_{pq}^{N} = \mathrm{W}^{G_N}(p,q),$
    \item $\xi_{pq}^{N} = \xi_{qp}^{N}= \text{Bernoulli}(\mathrm{W}^{G_N}(p,q))$ independently for $p,q \in V_N,$
\end{itemize}
where $\mathrm{W}^{G_N}$ converges in cut metric to graphon function $\mathrm{W}$.
\end{itemize}

From now on we will focus exclusively on the finite-dimensional case, i.e.,  $\mathfrak{X} = \{1,\dots,d\}$, equipped with a counting measure $\mu,$
$$\mathbb{H} := L^2((\mathfrak{X},\mu); \C) \equiv \C^{d}.$$
The space of density operators for a system of $N$-particles distributed on a graph $G_N = (V_N,E_n)$ is then described by the following set of density matrices

$$ \mathcal{S}(\mathbb{H}^{G_N}) := \Big\{\rho \in \mathcal{M}_{d^N}(\C); \rho \geq 0, \tr(\rho) = 1, \rho = \rho^{\dag} \Big\}.$$

Each particle is indirectly observed through an observable $L$. Consequently, the state of the system, conditioned upon the observation processes, is described by a stochastic differential equation with different driving noises depending on the type of detection measurement described below.

\paragraph{Homodyne detection}
The evolution of density matrix is described by matrix valued stochastic diffusive equation called Belavkin equation: 
\begin{align*}
\mathrm{d}\boldsymbol{\rho}_t^{N} &= -\iu{[\mathbf{H}^{},\boldsymbol{\rho}_{t}^{N}]}\mathrm{d}t
+ \sum_{q \in V_{N}} \Big({\bf L}_q\boldsymbol{\rho}_{t}^{N}{\bf L}^{\dag}_{q} - \frac{1}{2}\big\{{{\bf L}_q}{\bf L}^{\dag}_{q},\boldsymbol{\rho}_{t}^{N}\big\}\Big)\mathrm{d}t\\
& \!\!\!\!\!\!+\sqrt{\eta}\sum_{q \in V_{N}}\Big(\boldsymbol{\rho}_{t}^{N}{\bf L}^{\dag}_q+{\bf L}_q\boldsymbol{\rho}_{t}^{N} -\tr\big(({\bf L}_q + {\bf L}_q^{\dag})\boldsymbol{\rho}_{t}^{N}\big)\boldsymbol{\rho}_{t}^{N}\Big)\mathrm{d}W_t^{q},\\
\boldsymbol{\rho}_{0}^{N} &= \bigotimes_{{V_N}}\rho_{0}, \;\; \rho_{0} \in \mathcal{S}(\mathbb{H}).
\end{align*}
The $\eta \in (0,1]$ is the measurement efficiency, and the corresponding observation process for particle $q \in V_N$ is given by the following stochastic process
\begin{align*}
\mathrm{d}Y_t^{q} &= \mathrm{d}W_t^{q} + \sqrt{\eta}\tr\big( ({\bf L}_q + {\bf L}_q^{\dag})\boldsymbol{\rho}_{t}^{N}\big)\mathrm{d}t, \quad q \in V_N.
\end{align*}
\paragraph{Photon counting}
The evolution of the density matrix is described by the matrix valued stochastic jump equation called Jump Belavkin equation: 
\begin{align*}
\mathrm{d}&\boldsymbol{\rho}_t^{N} \!=\! -\mathrm{i}[\mathbf{H}^{},\boldsymbol{\rho}_{t-}^{N}]\mathrm{d}t
+\!\! \sum_{q \in V_{N}} \!\!\left({\bf L}_q\boldsymbol{\rho}_{t-}^{N}{\bf L}^{\dag}_{q}\! -\! \frac{1}{2}\left\{{{\bf L}_q}{\bf L}^{\dag}_{q},\boldsymbol{\rho}_{t-}^{N}\right\}\right)\!\!\mathrm{d}t\\
& \!+\!\sum_{q \in V_{N}}\!\!\left(\frac{{{\bf L}_q}\boldsymbol{\rho}_{t-}^{N}{\bf L}_q^{\dag}}{\tr\left({{\bf L}_q}\boldsymbol{\rho}_{t-}^{N}{\bf L}^{\dag}_{q}\right)}\!\! -\!\! \boldsymbol{\rho}_{t-}^{N}\!\!\right)\!\!\Bigg(\mathrm{d}N_t^{q}\! -\! \tr\left({{\bf L}_q}\boldsymbol{\rho}_{t-}^{N}{\bf L}^{\dag}_{q}\right)\mathrm{d}t \!\Bigg),\\
&\boldsymbol{\rho}_{0}^{N} = \bigotimes_{V_N}\rho_{0}, \; \rho_{0} \in \mathcal{S}(\mathbb{H}). 
\end{align*}

The (counting) observed Poisson processes $(N^{q}_t)_{t \geq 0}$ have stochastic intensities $\int_{0}^{t}\tr(\mathbf{L}_{q}^{\dag}\mathbf{L}_{q}\boldsymbol{\rho}^{N}_s)\mathrm{d}s$, so that the compensated processes $(N^{q}_t - \int_{0}^{t}
\tr(\mathbf{L}_{q}^{\dag}\mathbf{L}_{q}\boldsymbol{\rho}^{N}_s)\mathrm{d}s)_{t \geq 0}$ are martingales. 
It is worth noting that the driving noises $N^q$ depend on the unknown $\boldsymbol{\rho}_{t}^{N}$, while with a reformulation of $\mathrm{d}N_t^{q}\! -\! \tr\left({{\bf L}_q}\boldsymbol{\rho}_{t-}^{N}{\bf L}^{\dag}_{q}\right)\mathrm{d}t$ by means of Poisson measures, it becomes a standard SDE and the wellposedness is established. Nevertheless, identifying the corresponding mean-field limit and deriving the propagation of chaos are not trivial. For this technical reason, 
we consider the case where the operator $L$ is unitary (i.e. $L^{\dag} = L^{-1}$), such that the jump intensity is constant equal to one, and the equation becomes linear 
\begin{align*}
\mathrm{d}\boldsymbol{\rho}_t^{N} &= -\mathrm{i}[\mathbf{H}^{},\boldsymbol{\rho}_{t-}^{N}]\mathrm{d}t
+ \sum_{q \in V_{N}}\left({{\bf L}_q}\boldsymbol{\rho}_{t-}^{N}{\bf L}_q^{\dag} - \boldsymbol{\rho}_{t-}^{N}\right)\mathrm{d}N_t^{q}.
\end{align*}
An important notion in $N$-quantum body system  is the partial trace and can be viewed as quantum analogue of marginals for probability laws. 

\begin{definition}[Partial trace]
Set $V_k = \{\frac{1}{N},\dots,\frac{k}{N}\}$ for an integral operator $\boldsymbol{\rho}$ with integral kernel $\boldsymbol{\rho}(x;x') $ in $L^{2}(\mathfrak{X}^{N}),$ the partial trace over the set $V_N\setminus V_k$ is the operator $\rho^{V_k}$ denoted as follows
\begin{align*}
    \rho^{V_k} := \tr_{V_N\setminus V_k}(\boldsymbol{\rho}),
\end{align*}
where
\begin{align*}
    &\tr_{V_N\setminus V_k}(\boldsymbol{\rho})(x_{V_k};x'_{V_k}) \\ &= \int_{\mathfrak{X}^{V_N\setminus V_k}}\boldsymbol{\rho}\Big(x_{V_k},y_{V_N\setminus V_k};x'_{V_k},y_{V_N\setminus V_k}\Big)\mu(\mathrm{d}y_{V_N\setminus V_k}),
\end{align*}
with $x = (x_1,\dots,x_N), x_{V_k} = (x_1,\dots,x_k), x_{V_N\setminus V_k} = (x_{k+1},\dots,x_N).$

Similarly we can define the partial trace $\rho^{V} = \tr_{V_N\setminus V}(\boldsymbol{\rho})$ with respect to all variables outside the set $V \subset V_N.$
\end{definition}

\begin{remark}
The primary interest in models of open quantum systems subjected to indirect measurements is that they allow for the consideration of control situations through feedback (Markovian feedback in the terminology of MFG), thus the system's Hamiltonian can be extended into a time-dependent Hamiltonian of the following form:
\begin{align}
    \mathbf{H}^{}_{t} &= \frac{1}{N}\sum_{q \in V_N}\sum_{p > q} \xi^{N}_{pq}\mathbf{A}_{pq}\! + \! \sum_{q \in V_N}\Big(\mathbf{\tilde{H}}_{q} + \uc(\rho_t^{q})\hat{\mathbf{H}}_{q}\Big),\label{hamcont}
\end{align}
where:
\begin{itemize}
    \item The operator $\hat{\mathbf{H}}_{q} = \hat{\mathbf{H}}^{\dag}_{q}$  is the controlled Hamiltonian associated to particle $q.$
    \item Recall that $\rho^{q}$ represents the state of the particle $q$, which can be obtained by taking a partial trace over
the other particles i.e. $\rho^{q} := \rho^{\{q\}}.$
\item The scalar control $\uc \in \mathcal{U} := \big\{f \; \text{mesurable} \; |\; f : \mathcal{S}(\mathbb{H}) \to [-U, U] \big\}.$ 
\end{itemize}

All  results stated in this paper can be easily extended for Hamiltonian of the form \eqref{hamcont}, with a Lipschitzian control $\uc.$  
\end{remark}

\section{Mean-field and graphon models }
\label{sec:result}
\subsection{Propagation of chaos}
In order to reduce the complexity of representing $N$-interacting quantum particles, a natural approach is to seek simplification as $N$ becomes large. By doing so, we can disregard the impact of individual particles and instead focus on the effects of averaging, which arise from the law of large numbers.

The justification for the mean-field model goes back to the theory of propagation of chaos, initiated by Kac in \cite{kac56}, see \cite{chaintron22I,chaintron22II} for a recent and complete review. For the quantum case, this concept was initially formalized by Spohn in \cite{spohn1980kinetic}, see \cite{gottlieb03} for notion of quantum molecular chaos. The resulting limit equation is known as the non-linear Schr\"odinger equation or Hartree equation. Subsequently, Kolokoltsov extends these results by deriving the mean-field Belavkin  equations in a series of papers \cite{kolokoltsov2021law,kolokoltsov2022qmfg,kolokoltsov2021qmfgcounting}.

In the following, we assume the propagation of chaos. In essence, this implies that particles remain decoupled if they begin in a decoupled state, which can be summarized by the following intuitive idea : 
\begin{align*}
    \boldsymbol{\rho}_{0}^{N} = \rho_{0}^{\otimes N} \xrightarrow[]{\text{Propagation of chaos}} \boldsymbol{\rho}_{t}^{N}  \approx \bigotimes_{q \in V_N}\gamma_{t}^{q}. 
\end{align*}
As the graph is dense and converges with respect to cut metric to a graphon, we get
\begin{align*}
    \boldsymbol{\rho}_{t}^{N} \xrightarrow[N \to \infty]{} \bigotimes_{u \in I}\gamma_{t}^{u}.
\end{align*}
In this way, in the following subsection, we introduce the limit system and study its properties. Specifically, we investigate its well-posedness and stability characteristics. The rigorous derivation of the mean-field graphon systems will be presented in \cite{amini25quantum}.

\subsection{Quantum graphon system}
Relating the observations given by $ L$ to the evolution of particle $u$, we have the following limit equations for the individual particles:

\begin{itemize}

\item Homodyne detection:
\begin{align}
\mathrm{d}\gamma_t^{u} &= -\iu\Big[\tilde{H} + \int_{0}^{1}\mathrm{W}(u,v)A^{\mathbb{E}[\gamma_t^{v}]}\mathrm{d}v, \gamma_t^{u} \Big]\mathrm{d}t\nonumber\\
&+ \Big(L\gamma_t^{u}L^{\dag} - \frac{1}{2}\{LL^{\dag}, \gamma_t^{u}\}\Big)\mathrm{d}t\nonumber\\
& + \sqrt{\eta}\Big(L\gamma_t^{u} + \gamma_t^{u}L^{\dag} - \tr\big((L+L^{\dag})\gamma_t^{u}\big)\gamma_t^{u}\Big)\mathrm{d}W_t^{u}.\label{graphonbelvakinhomodyne}
\end{align}

\item Photon counting:
\begin{align}
\mathrm{d}\gamma_t^{u} &= -\iu\Big[\tilde{H} + \int_{0}^{1}\mathrm{W}(u,v)A^{\mathbb{E}[\gamma_{t-}^{v}]}\mathrm{d}v, \gamma_{t-}^{u} \Big]\mathrm{d}t\nonumber\\
&+\quad \Big(L\gamma_{t-}^{u}L^{\dag} - 
\gamma_{t-}^{u}\Big)\mathrm{d}N_t^{u},\label{graphonbelvakinjump}
\end{align}
\end{itemize}

where $\gamma_{0}^{u} = \rho_{0}$
and 
\begin{align*}A^{\rho}(x,x') &:= \int_{\mathfrak{X}^2}a(x,y;x',y')\overline{\rho(y,y')}\mu(\mathrm{d}y)\mu(\mathrm{d}y'),\\ &\; \forall (x,x') \in \mathfrak{X}^{2}, \;\; \forall \rho \in \mathcal{S}(\mathbb{H}).
\end{align*}

The observation process for the particle $u \in I$ will be:
\begin{itemize}
\item Homodyne detection: $$\mathrm{d}Y_t^{u} = \mathrm{d}W_t^{u} \!+\! \sqrt{\eta}\tr\big((L+L^{\dag})\gamma_t^{u}\big)\mathrm{d}t;$$

\item Photon Counting: $(N_{t}^{u})_{ t \geq 0}$.
\end{itemize}

\begin{remark}
In both types of detection, if we take expectation over trajectories $m_t := \mathbb{E}[\gamma_{t}]$, we get a new version of Lindblad equations:
\begin{align*}
\mathrm{d}m_t^{u} &= -\iu\Big[\tilde{H} + \int_{0}^{1}\mathrm{W}(u,v)A^{m_t^{v}}\mathrm{d}v, m_t^{u} \Big]\mathrm{d}t\\
&\qquad\qquad\qquad+ \Big(Lm_t^{u}L^{\dag} - \frac{1}{2}\{LL^{\dag}, m_t^{u}\}\Big)\mathrm{d}t.
\end{align*}
\end{remark}
\medskip

Now we can state the main results of this paper.

\medskip

\begin{theorem}[Existence and Uniqueness : Diffusive]
Let $T > 0 $. Then the system \eqref{graphonbelvakinhomodyne} is well posed, and for all $u \in I$, $\gamma^{u}$ takes values in the space $\mathcal{S}(\mathbb{H}).$
\end{theorem}

\begin{proof} The proof can be tackled using same techniques as in \cite{chalal23mean}.
Set $\mathcal{C}_{I} := \mathcal{C}\big([0,T],\mathcal{S}(\mathbb{H})^{I}\big)$ equipped with the following norm $\|\xi\|_{t} := \sup_{r \in [0,t]}\sup_{u \in I}\|\xi^{u}_{r}\|$. For each $\xi \in \mathcal{C}_{I}$,  we consider the following system SDE:
\begin{align*}
    \mathrm{d}\gamma_t^{\xi^{u}}&= -\iu\Big[\tilde{H}+\int_{0}^{1}\!\mathrm{W}(u,v)A^{\xi^{v}}\mathrm{d}v\!,\!\gamma_t^{\xi_{u}} \Big]\mathrm{d}t\\ &+ \Big(L\gamma_t^{\xi^{u}}L^{\dag}-\frac{1}{2}\{LL^{\dag},\gamma_t^{\xi^{u}}\}\Big)\mathrm{d}t\nonumber\\
& + \sqrt{\eta}\Big(L\gamma_t^{\xi^{u}} + \gamma_t^{\xi^{u}}L^{\dag} - \tr\big((L+L^{\dag})\gamma_t^{\xi^{u}}\big)\gamma_t^{\xi^{u}}\Big)\mathrm{d}W_t^{u}.
\end{align*}

From the parameterize system equations, define the mapping $$\Xi : \mathcal{C}_{I} \to
\mathcal{C}_{I} \; \text{by} \;\; 
\Xi(\xi) := \big((\mathbb{E}[\gamma_t^{\xi^{u}}])_{0 \leq t \leq T}\big)_{u \in I}.$$ 

The process $\gamma^{m_u}$ corresponds to the solution of \eqref{graphonbelvakinhomodyne} if and only if $m_u = \Xi_{u}(m)$. This is done by showing that there exists $k \in \N$ such that  $\Xi^{(k)}$ is contraction, namely for any $\tilde{\xi}, \hat{\xi} \in \mathcal{C}_{I}$
$$ \|\Xi^{(k)}(\tilde{\xi}) - \Xi^{(k)}(\hat{\xi}) \|_{T} < \|\tilde{\xi} - \hat{\xi}\|_{T},$$
see the complete proof in \cite{amini25quantum}.
\end{proof}

\begin{theorem}[Existence and Uniqueness : Jump]
Let $T > 0 $. Then the system \eqref{graphonbelvakinjump} is well posed,  and for all $u \in I$, $\gamma^{u}$ is valued in $\mathcal{S}(\mathbb{H}).$
\end{theorem}

\begin{proof}
The proof is similar to the previous one. The details will be provided in \cite{amini25quantum}.
\end{proof}

Now we prove the following stability result which is important to measure the distance between two state processes induced by different graphons.

\begin{theorem}[Stability : Diffusion]
Let  $(\tilde{\gamma}_t)_{t \geq 0} $ and $(\hat{\gamma}_{t})_{t \geq 0}$ be the solutions of the system \eqref{graphonbelvakinhomodyne} associated with graphon $\tilde{\mathrm{W}}$ and $\hat{\mathrm{W}}$, with the same initial conditions, then there exists  some constant $C := C(T,\eta, \|\tilde{H}\|,\|A\|,\|L\|)$ such that 
$$ \mathbb{E}\Big[\int_{I}\sup_{t \in [0,T]}\Big\|\tilde{\gamma}_{t}^{u} -\hat{\gamma}_{t}^{u}\Big\|^{2}\mathrm{d}u\Big] \leq C\|\tilde{\mathrm{W}} - \hat{\mathrm{W}}\|_{\square}. $$
\end{theorem}

\begin{proof} Set,
 $\Delta\gamma_{t}^{u} := \tilde{\gamma}_{t}^{{u}} - \hat{\gamma}_{t}^{{u}},$
     $\Delta\mathrm{W} := \tilde{\mathrm{W}} - \hat{\mathrm{W}}$ and 
$K:=L + L^{\dag}$.
Fix $u \in I $ and time $t \in [0,T]$. 
By Burkholder-Dabis-Gundy Inequality, Itô isometry and boundedness of $\tilde{\gamma},\hat{\gamma}$ 
\begin{align*}
&\mathbb{E}\big[\Big\|\tilde{\gamma_{t}}^{u} - \hat{\gamma_{t}}^{u}\Big\|^{2}\big]
\leq C\mathbb{E}\Big[\int_{0}^{t}\|\Delta\gamma_{s}^{u}\|^{2}\mathrm{d}s\Big] + C\dots\\
&\!\underbrace{\mathbb{E}\!\Bigg[\!\!\int_{0}^{t}\!\!\Big\|\!\!\int_{I}\!\big([\tilde{\mathrm{W}}(u,v)A^{\mathbb{E}[\tilde{\gamma}_{s}^{v}]},\tilde{\gamma}_{s}^{u}]\!-\![\hat{\mathrm{W}}(u,v)A^{\mathbb{E}[\hat{\gamma}_{s}^{v}]},\hat{\gamma}_{s}^{u}]\big)\mathrm{d}v\Big\|^{2}\!\!\mathrm{d}s\!\Bigg]}_{\mathcal{P}}\!\!.
\end{align*}
By straightforward triangular inequality, 
\begin{align*}
    \mathcal{P} &\leq \mathcal{P}_{1} + \mathcal{P}_{2} + \mathcal{P}_{3}\\
\mathcal{P}_{1} &= \mathbb{E}\Bigg[\int_{0}^{t}\Big\|\int_{I}[\tilde{\mathrm{W}}(u,v)A^{\mathbb{E}[\Delta\gamma_{s}]},\tilde{\gamma}^{u}_{s}] \mathrm{d}v\Big\|^{2}\mathrm{d}s\Bigg] \\ 
\mathcal{P}_{2} &= 
\mathbb{E}\Bigg[\int_{0}^{t}\Big\|\int_{I} [\Delta\mathrm{W}(u,v)A^{\mathbb{E}[\tilde{\gamma}_s^{v}]},\tilde{\gamma}^{u}_{s}] \mathrm{d}v\Big\|^{2}\mathrm{d}s\Bigg]\\
\mathcal{P}_{3} &= \mathbb{E}\Bigg[\int_{0}^{t}\Big\|\int_{I} [\tilde{W}(u,v)A^{\mathbb{E}[\tilde{\gamma}_{s}^{v}]},\Delta\gamma_{s}^{u}] \mathrm{d}v\Big\|^{2}\mathrm{d}s\Bigg].
\end{align*}

Using boundedness of density matrices,
\begin{align*}
\int_{I}\Big(\mathcal{P}_{1} + \mathcal{P}_{2}\Big)\mathrm{d}u \leq C\int_{I}\mathbb{E}\Big[\|\Delta\gamma_{s}^{u}\|^{2}\Big]\mathrm{d}u.
\end{align*}

\begin{align*}
    \int_{I}\mathcal{P}_{3}\mathrm{d}u &\leq C\int_{I}\mathbb{E}\Big[\big\|\tilde{\gamma}_{s}^{u}\big\|^{2}\Big]\Big\|\int_{I}\Delta\mathrm{W}(u,v)\big(A^{\mathbb{E}[\tilde{\gamma}_{s}^{u})]}\big)\mathrm{d}v\Big\|\mathrm{d}u\\
    &\leq C\int_{I}\Big\|\int_{I}\Delta\mathrm{W}(u,v)\big(A^{\mathbb{E}[\tilde{\gamma}_{s}^{u})]}\big)\mathrm{d}v\Big\|\mathrm{d}u\\
    &\leq C\big\|\mathcal{T}_{\tilde{\mathrm{W}} - \hat{\mathrm{W}}}\big\|_{\text{op}} \leq C\big\|\tilde{\mathrm{W}} - \hat{\mathrm{W}}\big\|_{\square},
\end{align*}
where for the last inequality we use Lemma \ref{estimatecutop}.
By combining all the estimates,
\begin{align*}
\int_{I}\mathbb{E}\Big[\big\|\Delta{\gamma}_{t}^{u}\big\|^{2}\Big]\mathrm{d}u &\!\leq \!C\Bigg(\!\!\int_{I}\int_{0}^{t}\mathbb{E}\Big[\big\|\Delta{\gamma}_{s}^{u}\big\|^{2}\Big]\mathrm{d}s\mathrm{d}u\! + \!\big\|\Delta{\mathrm{W}}\big\|_{\square} \!\!\Bigg). 
\end{align*}

We conclude by Grönwall's lemma.
\end{proof}

Similarly, the same theorem can be stated for the Poissonian case.
\begin{theorem}[Stability : Jump]
Let  $(\tilde{\gamma}_t)_{t \geq 0} $ and $(\hat{\gamma}_{t})_{t \geq 0}$ be the solutions of the system \eqref{graphonbelvakinjump}  associated with graphon $\tilde{\mathrm{W}}$ and $\hat{\mathrm{W}}$, with the same initial conditions, then there exists  some constant $C := C(T,\eta, \|\tilde{H}\|,\|A\|,\|L\|)$ such that 
$$ \mathbb{E}\Big[\int_{I}\sup_{t \in [0,T]}\Big\|\tilde{\gamma}_{t}^{u} -\hat{\gamma}_{t}^{u}\Big\|^{2}\mathrm{d}u\Big] \leq C\|\tilde{\mathrm{W}} - \hat{\mathrm{W}}\|_{\square}.$$
\end{theorem}

\section{Numerical simulation }
\label{sec:numerical}
To highlight  the main nature of  this graphon quantum systems, we consider a toy model of graphon qubit  systems, where $\tilde{H} = \sigma_{z}, \hat{H} = \sigma_{x}, L = \sigma_{z}$, and $A$ denotes the exchange  the
interaction operator between qubits given by  $A = a^{\dag}_{1}\otimes a_2 + a^{\dag}_{2}\otimes a_{1}$. This operator represents the exchange
of a single photon between two qubits, where $a^{\dag}$ and $a$ are the creation and annihilation operators, see \cite{chalal23mean} for more details. The graphon function $\mathrm{W}$ is given as
\begin{align*}
    \mathrm{W}(u,v) &= 1 \; \text{if } (u-\frac{1}{2})(\frac{1}{2} - v) \geq 0,\\
    &= 0 \; \text{otherwise}.
\end{align*}

The graphon systems is reduced to two classes of particles, denoted by $0$ and $1$, whose dynamics is given by the following stochastic differential equations: for $j=0,1$,
\begin{align*}
    \mathrm{d}\gamma_t^{j} &= -\iu[\sigma_{z} + \uc_{j}(\gamma^{j})\sigma_x + \frac{1}{2}A^{\mathbb{E}[\gamma_t^{1-j}]},\gamma_t^{j}]\mathrm{d}t + \\
    &\big(\sigma_z\gamma_t^{j}\sigma_{z} - \gamma_{t}^{j}\big)\mathrm{d}t + (\sigma_{z}\gamma_{t}^{j} + \gamma_{t}^{j}\sigma_{z} - 2\tr(\sigma_{z}\gamma_{t}^{j})\gamma_{t}^{j})\mathrm{d}W_t^{j},
\end{align*}
where $\uc_{j}(\gamma^{j}) = -8\iu\tr\Big([\sigma_{x},\gamma^{j}]\tau^{j}\Big) \!+\! 5\big(1 - \tr(\gamma^{j}\tau^{j})\big)$ and 
$$\gamma^{j}_t \!\!:=\!\!\frac{1}{2}\begin{pmatrix}
    1+z^j_t & x^j_t -\iu y^j_t \\ x^j_t +\iu y^j_t & 1-z^j_t 
\end{pmatrix}, \tau^{0} \!\!:=\!\! \begin{pmatrix}
    1 & 0 \\ 0 & 0 
\end{pmatrix}, \tau^{1}\!\!:= \!\!\begin{pmatrix}
    0 & 0 \\ 0 & 1 
\end{pmatrix}.$$

Straightforward computation in Pauli basis yields 
{\small\begin{align*}
\small\mathrm{d}x_t^{j}\!&=\! \small \Big( - y_t^{j} - x_t^{j} + z_t^{j}\mathbb{E}[y_t^{1-j}]\Big)\mathrm{d}t - x_t^{j}z_t^{j}\mathrm{d}W_t^{j},\\
\small\mathrm{d}y_t^{j}\!&=\! \small\Big( x_t^{j} - {y_t}^{j} + \uc_j(\gamma_t^{j})z_t^{j} -z_t^{j}\mathbb{E}[x_t^{1-j}]\Big)\mathrm{d}t + y_t^{j}z_t^{j}\mathrm{d}W_t^{j},\\
\small\mathrm{d}z_t^{j}\!&=\! \small\Big(\!-\!\uc_j(\gamma_t)x_t^{j} + y_t^{j}\mathbb{E}[x_t^{1-j}] \!+\! x_t^{j}\mathbb{E}[y_t^{1-j}]\Big)\mathrm{d}t + \big(1 \!-\! z_t^2)\mathrm{d}W_t^{j}. 
\end{align*}}

\begin{figure}[h!]
     \centering
  \includegraphics[width=\linewidth]{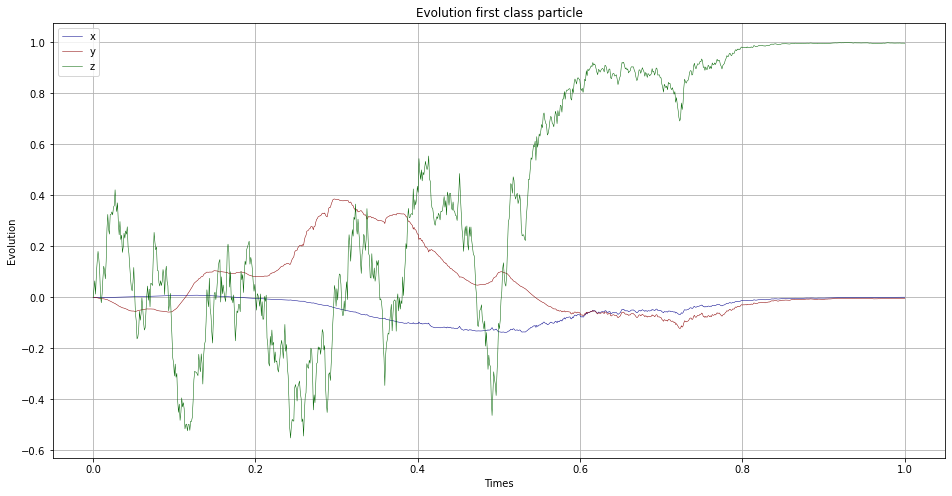}
  \caption{\footnotesize{Evolution of $x^0_t$(\textcolor{blue}{b}),$y^0_t$(\textcolor{red}{r}),$z^0_t$(\textcolor{green}{g}) component for the first class}}
    \label{fig:1}
  \end{figure}

  \begin{figure}[h!]
     \centering
  \includegraphics[width=\linewidth]{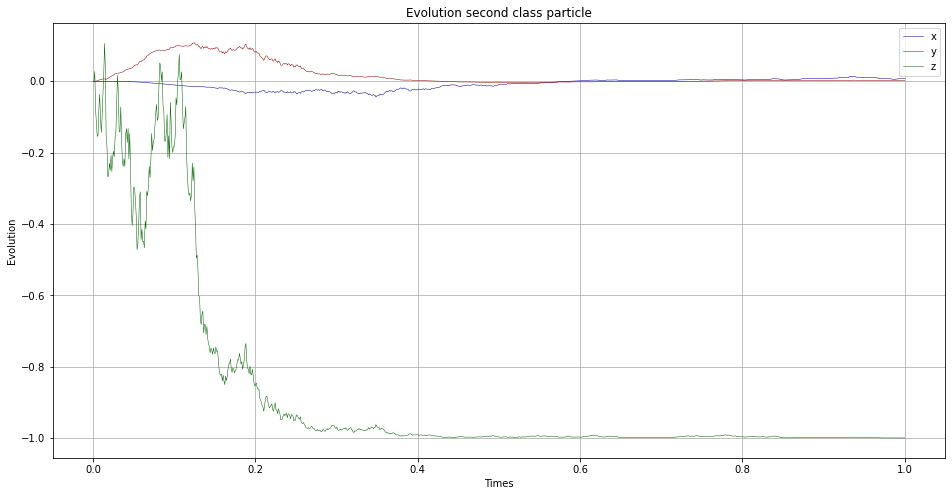}
  \caption{\footnotesize{Evolution of $x^1_t$(\textcolor{blue}{b}),$y^1_t$(\textcolor{red}{r}),$z^1_t$(\textcolor{green}{g}) component for the second class}}
    \label{fig:2}
  \end{figure}
As shown in Figures \ref{fig:1} and \ref{fig:2}, the feedback $\uc_j(\gamma^j)$ stabilizes $\gamma^j$ toward the target state $\tau^j.$

\section{CONCLUSIONS}
\label{sec:conc}
In this paper, we have derived a novel system of Belavkin equations for observed quantum particles with a graphon interaction. We studied several characteristics of the limit system. This new framework, coupled with optimal control problems, allows the consideration of quantum graphon stochastic games. In future work, we will provide more details on the results of this paper and in particular demonstrate the convergence of the density matrix towards the graphon system.  
\addtolength{\textheight}{-12cm}   




\bibliographystyle{unsrt}
\bibliography{Bibliography}

\end{document}